\documentclass[12pt]{article}
\usepackage{epsfig}
\usepackage{amssymb}

\newtheorem{theorem}{Theorem}
 \newtheorem{lemma}[theorem]{Lemma}

\newtheorem{definition}[theorem]{Definition}

\newenvironment{proof}{{\it Proof:\/}}{$\Box$\vskip 0.08in}

\begin{document}
\thispagestyle{empty}
\
\vspace{0.1in}
 \begin{center}
 {\LARGE\bf A note on the Lickorish-Millett-Turaev formula for the Kauffman
polynomial\footnote{{\it Proceedings of the  American  Mathematical Society}, Volume 121, Number 2, 
June 1994, 645-647.}\footnote{Received by the editors 
 January 5, 1993.} }

\end{center}
\vspace*{0.2in}
 \begin{center}
{\bf  \large                   J\'ozef H.~Przytycki}  
\end{center}
\centerline{ (Communicated by Ronald Stern)}
\ \\
{\footnotesize ABSTRACT.
We use the idea of expressing a nonoriented link as a sum of all oriented links
corresponding to the link to present a short proof of the Lickorish-Millett-Turaev formula for
the Kauffman polynomial at $z= -a- a^{-1}$. Our approach explains the observation made by
Lickorish and Millett that the formula is the generating function for the linking number of a sublink of the given
link with its complementary sublink.} 

\vspace*{0.2in}

We have conjectured \cite{Mor} and partially proved (in April 1986) that the Kauffman 
polynomial of a knot at $z= -a- a^{-1}$ is equal to one. Lickorish and
Millett solved the conjecture (in August 1986) and found the formula for the
Kauffman polynomial at $z= -a- a^{-1}$ for any link \cite{L-M} (see also  \cite{Lip}).
The formula was also independently discovered, in a more general context, by Turaev
 \cite{Tur-1}. We give here the ``calculationfree" approach to the formula using the idea
of presenting a nonoriented link as a sum of all corresponding oriented links.

To fix the notation we recall the standard definition of the Kauffman polynomial, 
$F_L(a, z)$, of unoriented framed links in $S^3$: \
$F_L(a, z)\in Z[a^{\pm 1},z^{\pm 1}]$ is uniquely determined by the following properties:
\begin{enumerate}
\item[(1)] $F_{T_1}=1$ where $T_1$ is the unknot with $0$-framing.
\item[(2)] $F_{L^{(1)}}= aF_L$, where $L^{(1)}$ is the framed link obtained from $L$ by adding
 a full right-hand twist to the framing of $L$.
\item[(3)] $F_{L_+}+ F_{L_-} = z(F_{L_0}+ F_{L_{\infty}})$, where $L_+,L_-, L_0$, and $L_{\infty}$ denote 
 four unoriented links which are the same, except in a small ball where they look as
in Figure 1 on the next page, and the framing, in the formula, is assumed to be
vertical to the plane of the projection.
\end{enumerate}

The Kauffman polynomial of oriented links, $F_L(a, z)$, which is an invariant
of ambient isotopy, is defined as the Kauffman polynomial of the corresponding
unoriented $0$-framed link (i.e., the framing is given by the Seifert surface of the
oriented link). The Kauffman polynomial of an unoriented diagram $D$, which
is an invariant of regular isotopy, is equal to the Kauffman polynomial of the
link given by the diagram with the vertical framing.

 We will derive the formula (of Lickorish, Millett, and Turaev) from the idea
of presenting an unoriented link as a sum of oriented links corresponding to it
(the formula is obtained without guessing it in advance).
\\ \ \\

\centerline{\psfig{figure=L+L-L0Linf.eps,height=2.3cm}}
\begin{center}
Figure 1
\end{center}

\begin{definition}\label{Definition 1}  Let ${\mathcal S}_{fr}$ denote the set of ambient isotopy classes of 
framed unoriented links in $S^3$ and $\vec {\mathcal S}_{fr}$ the set of ambient isotopy classes 
of framed oriented links in $S^3$. Further, let $R = Z[a^{\pm 1}]$ and $RX$ denote the free $R$-module 
with basis $X$. We define a ``transfer" map $\tau: R{\mathcal S}_{fr} \to R\vec{\mathcal S}_{fr}$ by 
$\tau(L) = \sum_{L'\in OR(L)}L',$ where $OR(L)$ is the set of all orientations of $L$. 
For example, $\tau({\parbox{0.4cm}{\psfig{figure=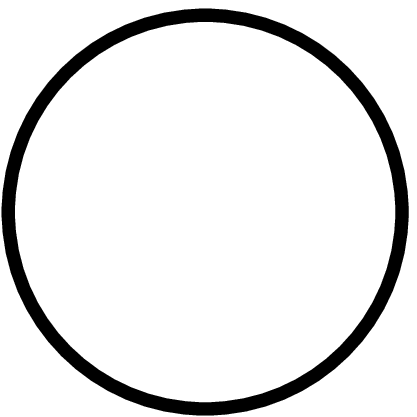,height=0.4cm}}}) 
= {\parbox{0.4cm}{\psfig{figure=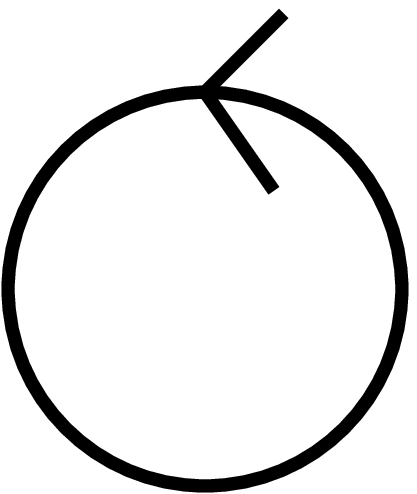,height=0.4cm}}} + 
{\parbox{0.4cm}{\psfig{figure=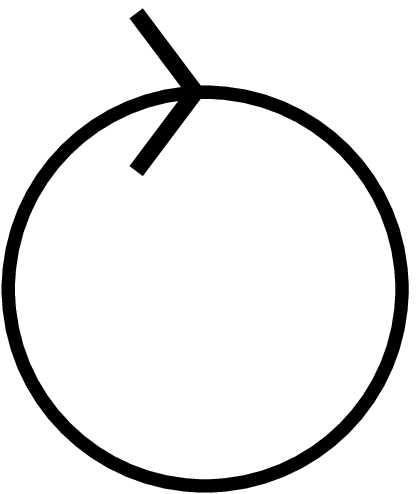,height=0.4cm}}}$.

\end{definition}

Now consider the following very simple invariant of framed oriented links.

\begin{definition}\label{Definition 2} $g: R\vec{\mathcal S}_{fr} \to Z[a^{\pm 1}]$ is defined 
by  $g(L)= (-1)^{com L}a^{fr(L)}$, where 
$com(L)$ is the number of components of $L$ and $fr(L)$ is the framing number
of $L$; that is, the number of right-hand twists that must be removed from the
framing of $L$ to reach the $O$-framing.
\end{definition}

We will analyze the composition $g\tau$ and first show that it is the Kauffman
polynomial at $z= -a- a^{-1}$ multiplied by $-2$. Then we will easily evaluate 
$F(a, -a- a^{-1})$ using its relation to $g\tau$.

\begin{lemma}\label{Lemma 3} 
\begin{enumerate}
\item[(a)] $g\tau(L_+) + g\tau(L_-) = (-a- a^{-1})(g\tau(L_0)+ g\tau(L_{\infty}))$.
\item[(b)] $g\tau(L) = -2aF_L(a, -a- a^{-1})$, where $L$ is any unoriented framed link.
\end{enumerate}
\end{lemma}

\begin{proof}
We must consider two cases:
\begin{enumerate}
\item[(1)] (the case of a self-crossing) In this case $L_0$ or $L_{\infty}$ (say $L_0$) 
has one more component than the other three links involved in the relation. So $L_0$ has 
twice as many orientations as the other links. Half of these orientations are the
same as in $L_{\infty}$, and the other half agrees with that of $L_+$. The formula in
Lemma \ref{Lemma 3}(a) follows easily when one observes that if an oriented diagram $D$ 
has the vertical framing then $fr(D) = Tait(D)$, where $Tait(D)$ is the algebraic
sum of the signs of the crossings of $D$.
\item[(2)] (the case of a mixed crossing) In this case $L_+$ (and $L_-$) has one more
component than $L_0$ (and $L_{\infty}$) and therefore twice as many possible orientations. 
As in (1) formula (a) follows almost immediately.\\
(b) The formula in (a) agrees with that for the Kauffman polynomial $F(a, -a- a^{-1})$. 
Furthermore in both cases, if the framing of a link is changed
by adding to it a positive twist then the invariant is multiplied by $a$.
Therefore, we must compare our two invariants for the $O$-framed unknot, $T_1$. Then
$g\tau(T_1)=-2$, and $F_{T_1}(a,-a- a^{-1})=1$. Lemma \ref{Lemma 3}(b) follows. 
\end{enumerate}
\end{proof}

\begin{lemma}\label{Lemma 4}
Let $L$ be an oriented link in $S^3$ with $O$-framing and $p(L)$ the same
framed link but without an orientation. Then
$$ g\tau(p(L))= (-1)^{com (L)}\sum_{S\subset L}a^{-4lk(S,L-S)}$$
where the summation is taken over all sublinks $S$ of $L$ (including $S= \emptyset$) and
$lk(S,L-S)$ denotes the global linking number between $S$ and $L-S$.
\end{lemma}
\begin{proof}
If $S$ is a sublink of $L$ and $L$ is the oriented framed link obtained from
$L$ by changing the orientation of the components in $S$ but keeping the framing
of $L$ then $fr(L_S) = -4lk(,S, L-S)$ ; in particular $fr(L_{\emptyset}) = 0$. From the above
observation the formula for $g\tau(p(F)) = g(\sum_{S\subset L}L_S) $ follows immediately.
\end{proof} 

As a corollary we have the Lickorish-Millett-Turaev formula.

\begin{theorem}\label{Theorem 5}
$$F_L(a,-a-a^{-1})= ((-1)^{com(L)-1})/2\sum_{S\subset L}a^{-4lk(S,L-S)}.$$
\end{theorem}
\begin{proof}
We combine Lemmas \ref{Lemma 3} and \ref{Lemma 4}.
\end{proof}

The idea used in the note was applied for the first time in \cite{Gol} and ascribed
to Dennis Johnson (see also \cite{Tur-2,H-P-1}). There are other applications of the idea in
the theory of skein modules \cite{H-P-2}, which we hope to describe in the future.\footnote{Added for 
e-print: Consult \cite{Prz} for an extensive overview of skein modules.}

\ \\
 \ \\

\ \\
Department of Mathematics and Computer Science\\
Odense University, Campusvej 55, DK-5230,\\
Odense M, Denmark\\
E-mail address: Jozef@imada.ou.dk
\ \\ \ \\
Current address (added for e-print):\\
Department of Mathematics\\
George Washington University  \\
Washington, DC 20052 \\
e-mail: przytyck@gwu.edu
\end{document}